\documentclass[letterpaper, 10pt, conference]{ieeeconf}
\IEEEoverridecommandlockouts        
\title{\LARGE \bf Network Topology Change Identification from Inverse Covariance Matrices}

\title{\LARGE \bf Identifying Edge Changes in Networks from Inverse Covariance Matrices}

\title{\LARGE \bf Edge Changes Identification in Infrastructure Networks: A Sparse Total Least Squares Approach}

\title{\LARGE \bf Structure Learning via ADMM in Networks obeying Conservation Laws}

\author{Rohith Reddy Mada and Rajasekhar Anguluri
	\thanks{The authors are with the Department of Computer Science and Electrical Engineering, University of Maryland, Baltimore County, MD 85281, USA (e-mails: \href{rohithm4@umbc.edu}{rohithm4@umbc.edu} and \href{rajangul@umbc.edu}{rajangul@umbc.edu}). The first author is a second-year master's student, and the last author is an assistant professor.}}

\usepackage{graphics,color}
\usepackage[outdir=./]{epstopdf}
\usepackage{amsmath}
\usepackage{amssymb}
\usepackage{mathrsfs}
\usepackage{subfigure}
\usepackage[font=small,skip=5pt]{caption}
\usepackage{url}
\usepackage{float}
\usepackage{dsfont}
\usepackage{booktabs}
\usepackage{array}
\usepackage[table]{xcolor}
\let\labelindent\relax
\usepackage{enumitem}
\usepackage{optidef} 
\usepackage{float}
\usepackage[colorlinks=true, linkcolor=black, citecolor=black, urlcolor=black]{hyperref}
\usepackage{makecell}

\usepackage{bbold} 

\newtheorem{theorem}{\bf \emph{Theorem}}[section]

\newtheorem{proposition}[theorem]{Proposition}

\newtheorem{remark}{Remark}

\newcommand{\transpose}{\mathsf{T}} 

\newcommand*{\QEDW}{\hfill\ensuremath{\square}}

\newcommand{\vertiii}[1]{{\left\vert\kern-0.25ex\left\vert\kern-0.25ex\left\vert #1 \right\vert\kern-0.25ex\right\vert\kern-0.25ex\right\vert}}
\DeclareMathOperator{\Tr}{Tr}

\newcommand\oprocendsymbol{\hbox{$\square$}}
\newcommand\oprocend{\relax\ifmmode\else\unskip\hfill\fi\oprocendsymbol}

\usepackage[linesnumbered,ruled,vlined]{algorithm2e}

\usepackage{tabstackengine}
\setstackEOL{\cr}

\begin{document}
	\maketitle
	
	\thispagestyle{empty} \pagestyle{empty}
	
\begin{abstract} Learning the edge connectivity structure of networked systems from limited data is a fundamental challenge in many critical infrastructure domains, including power, traffic, and finance. Such systems obey steady-state conservation laws: \( x = L^{\ast} y \), where \( x \) and \( y \in \mathbb{R}^p \) represent injected flows (inputs) and potentials (outputs), respectively. The sparsity pattern of the \( p \times p \) Laplacian \( L^{\ast} \) encodes the underlying edge structure. In a stochastic setting, the goal is to infer this sparsity pattern from zero-mean i.i.d. samples of \( y \).  

Recent work by \cite{rayas2022learning} has established statistical consistency results for this learning problem by considering an $\ell_1$-regularized maximum likelihood estimator. However, their approach did not develop a scalable algorithm but relies on solving a convex program via the CVX package. To address this gap, we propose an alternating direction method of multipliers (ADMM), which is transparent and fast. A key contribution is to demonstrate the role of an algebraic matrix Riccati equation in the primal update step of ADMM. Numerical experiments on a host of synthetic and benchmark networks, including power and water systems, show the efficiency of our method.

	\end{abstract}

	\section{Introduction}

Critical infrastructure systems—such as power, water, and traffic networks—are modeled as graphs, where conservation laws enforce that nodal flows across the edges of the network are neither created nor lost. In steady state, these conservation laws reduce to linear relation \( x = L^* y \), where \( L^* \in \mathbb{R}^{p \times p} \) is a weighted symmetric Laplacian encoding the network's edge structure through its off-diagonal entries \cite{van2017modeling}. The $p$-dimensional vectors $x$ and $y$ denote nodal injections and potentials, respectively.
 
In practice, the edge connectivity structure is unknown and must be learned from data \cite{deka2023learning, deka2020graphical, cavraro2021learning}. This entails recovering the sparsity pattern of the unknown (true) \( L^* \) using finite samples or summary statistics of the injection–potential pairs \( \{ x, y \} \). Assuming random injection $x$, and given the sample covariance matrix \( S \) of \( y \),~\cite{rayas2022learning} proposed estimating the network structure by solving the optimization problem:
\begin{align}\label{eq: log-det}
\hat{L}=\underset{L \in \mathcal{S}_p^{++}}{\text{arg\,min  }} \Tr(SL^2) - 2\log\det(L) + \lambda \| L \|_{1,\text{off}},
\end{align}
where \( \mathcal{S}_p^{++} \) is the set of real-valued $p\times p$ symmetric positive definite matrices and $\|L\|_{1,\text{off}}$ is the $\ell_1$-norm applied to  off-diagonals of $L$. The edge structure can then be inferred from the sparsity pattern of the estimate \( \hat{L} \) (see Section~\ref{sec: prelims}).
 
Ref.~\cite{rayas2022learning} have established many properties of the problem in \eqref{eq: log-det}. Chief among them are that the problem is convex and that its statistical complexity depends on the maximum degree of the network. Further, \eqref{eq: log-det} is related to many graph learning problems (see related work in~\cite{rayas2022learning, rayas_stationary2025}). Thus, developing simple and fast iterative algorithms for solving~\eqref{eq: log-det} is of great interest. This motivates our use of the alternating direction method of multipliers (ADMM).  
We highlight that although~\eqref{eq: log-det} resembles the sparse inverse covariance estimation (ICE) problem in graphical models, a key difference lies in the trace $\operatorname{Tr(\cdot)}$. In~\eqref{eq: log-det}, \( \operatorname{Tr}(\cdot) \) is quadratic in the variable \( L \), whereas in ICE it is linear. This subtle yet significant distinction makes the development of an ADMM for~\eqref{eq: log-det} challenging and interesting. Our contributions are: 




{\color{black} \begin{enumerate}
    \item For the unregularized case ($\lambda = 0$), we show that the global minimizer $\hat{L}$ of \eqref{eq: log-det0} is equal to the inverse of the positive semi-definite square root matrix of the sample covariance $S$, if inverse exists (see Prop.~\ref{prop: unrestricted MLE}).
    \item We propose a simple ADMM algorithm to numerically solve the structure learning problem in \eqref{eq: log-det0} for any $\lambda \geq 0$. We show that the dual and Lagrange multiplier update subproblems can be solved in closed form.
    \item The primal subproblem of our ADMM algorithm does not admit closed-form solution, as it involves solving an algebraic Riccati equation (NARE). Therefore, we resort to a Newton's method to solve the NARE. In Remark \ref{rem: ADMM comparision}, we clarify why this is not an issue in the closely related sparse ICE problem in Gaussian graphical models.
\end{enumerate}
Numerical results on many synthetic and real-world benchmark networks demonstrate the superior efficiency (in terms of iterations and execution time; see Table.~\ref{tab:graph_comparison}) of our ADMM method compared to the CVX-based approach in \cite{rayas2022learning}. The limitations of the proposed approach and future work are discussed toward the end of the paper.


{\color{black}\textit{Related Research}: There is a wide literature on developing ADMM-based methods for matrix-valued statistical learning problems (see \cite{boyd2011distributed, ma2014gentle, lin2022alternating}). Structure learning in conservation networks has been studied in various contexts, which broadly fall into two categories. The first is the mean-based approach, where both \(x\) and \(y\) are observed, reducing the learning problem to a sparse regression task \cite{ardakanian2019identification, brouillon2021bayesian, babakmehr2016smart}. The second category assumes \(x\) and \(y\) are random, but only \(y\) is observed. Here, the network is inferred from the (sample) covariance of \(y\) (see \cite{deka2020graphical, rayas2023differential, rayas2022learning, rayas2024learning}). Notably, \cite{deka2023learning} and \cite{cavraro2021learning} surveys structure learning in power systems. 

\section{Preliminaries and Problem Setup}\label{sec: prelims}
Consider an undirected, connected graph with \( p + 1 \) nodes, where the node set is \( V = \{0, 1, 2, \ldots, p\} \) and the edge set \( E \subset V \times V \). Each edge \( (i, j) \in E \) is assigned a non-negative weight \( a_{ij} \geq 0 \). Let \( x, y \in \mathbb{R}^{p+1} \) denote the vectors of nodal injections and potentials, respectively. The conservation law is represented by \( x = L_{\text{full}}^* y \), where \( L_{\text{full}}^* \in \mathbb{R}^{(p+1) \times (p+1)} \) is the full Laplacian matrix defined by
\[
[L^*_\text{full}]_{ij} = 
\begin{cases}
- a_{ij}, & \text{if } i \ne j, \\
\sum_{k \ne i} a_{ik}, & \text{if } i = j.
\end{cases}
\]
Thus, \( (i, j) \in E \) iff \( [L^*_\text{full}]_{ij} \ne 0 \). The applicability of linear conservation laws and their extensions is elegantly discussed in~\cite{van2017modeling}. Fig.\ref{eq: sparsity of L} illustrates the sparsity pattern of $L^*$.

Following~\cite{rayas_stationary2025, rayas2022learning, deka2020graphical}, we work with the reduced Laplacian matrix \( L^* \), obtained by removing node \( 0 \) and its incident edges from the graph. In particular, \( L^* \in \mathbb{R}^{p \times p} \) is formed by deleting the row and column corresponding to node \( 0 \) from \( L_{\text{full}}^* \). The reduced matrix is symmetric and positive definite, and therefore invertible, ensuring that \( L^* \), or at least its sparsity pattern, can be recovered from data.

  	\begin{figure}
		\centering
		\includegraphics[width=1.0\linewidth]{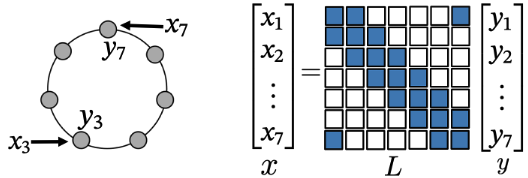}
		\caption {\small {\bf sparsity pattern of a Laplacian encoding edge structure}: (left) A network consisting of seven nodes labeled in a counterclockwise direction, with node 1 positioned to the left of node 7. Each node is assigned a pair \((x_i, y_i)\), where arrows represent the injection of flows \(x_i\). (right) The indices of blue-colored cells in \( L^* \) indicate the presence of edges between nodes identified by these indices, while white-colored cells signify the absence of edges.}\label{eq: sparsity of L}
        \end{figure}
\setlength{\textfloatsep}{2pt plus 1.0pt minus 2.0pt}
  
\textbf{Structure learning problem }\cite{rayas2022learning}: \textit{Estimate the sparsity pattern of the unknown true Laplacian matrix \( L^* \) using \( N \) i.i.d. samples of the potential vector \( y \), assuming that the nodal injection vector \( x \) is a zero-mean Gaussian random vector with a known, inverse covariance matrix \( \Theta_x \succ 0 \).}

The Gaussian random injection assumption is common in the literature on network structure learning and is motivated by studies in power systems, where the injections are random load fluctuations \cite{deka2023learning, cavraro2021learning}. Thus, the available data consists of \( (\Theta_x, \{y_1,\ldots,y_N\}) \), but not the actual flow injections \( x \).

\subsection{An $\ell_1$-regularized Maximum Likelihood Estimator}
Ref.~\cite{rayas2022learning} framed the structure learning problem as an \( \ell_1 \)-regularized maximum likelihood estimation (MLE), leading to the optimization problem given in~\eqref{eq: log-det}. For completeness, we provide a high-level overview of the MLE derivation.

Let $x\sim \mathcal{N}(0,\Theta_x^{-1})$. Because $L^*$ is invertible, the conservation law can be written as $y = (L^*)^{-1}x$. From Gaussianity of $x$, it follows that $y\sim \mathcal{N}(0,\Theta_y^{-1})$, where
$\Theta_y=L^*\Theta_xL^*$. Consider $N$ i.i.d samples $\{y_,\ldots,y_N\}$ and define the $N\times N$ sample covariance matrix $S=N^{-1}\sum_{t=0}^Nyy^\transpose$. 
The log-likelihood of the data $y_1,\ldots,y_N$ is given by
\begin{align}\label{eq:log-likelihood-L}
    g(L^*) \propto 2 \log\det(L^*) - \operatorname{Tr}(S L^* \Theta_x L^*) , 
\end{align}
where $\propto$ is the proportionality symbol. While \( g(L^*) \) cannot be evaluated since \( L^* \) is unknown, we may instead consider \( g(L) \) for a generic positive definite matrix \( L \). The MLE then seeks a matrix \( L \succ 0 \) that maximizes \( g(L) \).

We adopt the minimization framework to arrive at the \emph{$\ell_1$-norm regularized maximum likelihood estimator (MLE)}: 
\begin{align}\label{eq: log-det0}
\hat{L}=\underset{L \in \mathcal{S}_p^{++}}{\text{arg\,min  }} \Tr(SL \Theta_x L) - 2\log\det(L) + \lambda \| L \|_{1,\text{off}}. 
\end{align}
Here, \( \mathcal{S}_p^{++} \) denotes the set of \( p \times p \) real symmetric positive definite matrices; \( \| L \|_{1,\text{off}} = \sum_{i \neq j} |L_{ij}| \) is the \( \ell_1 \)-norm applied to the off-diagonal entries of the optimization variable \( L \); and \( \lambda \geq 0 \) is a tuning parameter. The inclusion of the \( \ell_1 \)-penalty promotes sparsity in the estimate \( \hat{L} \). The formulation in \eqref{eq: log-det0} is general version of \eqref{eq: log-det} where in $\Theta_x=I$. 

Ref.~\cite{rayas2022learning} showed that the regularized MLE problem in~\eqref{eq: log-det0} is convex, and that the solution \( \hat{L} \) is unique for all \( \lambda > 0 \), even when \( N < p \). Note that the optimization domain in~\eqref{eq: log-det0} is \( \mathcal{S}_p^{++} \), not the set of reduced Laplacian matrices, which form a strict subset of \( \mathcal{S}_p^{++} \). Recent work~\cite{kumar2019structured} has addressed optimization over Laplacians directly, but that is beyond the scope of this paper. 

Before presenting an ADMM-based algorithm to compute \( \hat{L} \), we first state an interesting result for the unrestricted MLE \( \lambda = 0 \), which is new.




\begin{proposition}\label{prop: unrestricted MLE}
    Let $\lambda=0$ and assume that $N>p$. Then with probability one, the unrestricted MLE (the unique global minimizer of \eqref{eq: log-det0}) is $\hat{L}=S^{-1 / 2}\left(S^{1 / 2} \Theta_x^{-1} S^{1 / 2}\right)^{1 / 2} S^{-1 / 2}$.
    \end{proposition}
\begin{proof}
    The loss function in \eqref{eq: log-det0} (for $\lambda=0$) is convex and differentiable. The zero-gradient condition is $LSL=\Theta_x^{-1}$. 
Because $N > p$ with probability one, $S\succ 0$ \cite{dykstra1970establishing}. Thus, there exists a unique positive definite square root $S^{1/2}$ satisfying $S = S^{1/2} S^{1/2}$. It can be verified that the matrix \( \hat{L} \), as stated in the proposition, satisfies the zero-gradient condition, and is the unique positive definite solution to the unrestricted MLE (see~\cite[Exercise 1.2.13]{bhatia2009positive}).
\end{proof}

Let $\Theta_x = I$. From Proposition \ref{prop: unrestricted MLE}, the unrestricted MLE is $\hat{L} = S^{-1/2}$. This estimate is intuitive, as in the population case, we have $\Theta^*_y = (L^*)^2$, which results in $L^* = (\Theta_y)^{1/2}$. Therefore, in the sample setting, we expect that the estimate of $L^*$ should equal the inverse square root of the sample covariance $S$, which is what MLE provides. Unfortunately, for $\lambda > 0$, there is no simple closed-form expression for $\hat{L}$, unlike the result in Proposition \ref{prop: unrestricted MLE}. This is because of the non-differentiability of the $\ell_1$-regularizer.




\smallskip  
\textbf{Problem}: Develop an efficient ADMM algorithm to numerically solve the maximum likelihood problem in \eqref{eq: log-det0} for any $\lambda \geq 0$, with either $N < p$ or $N \geq P$.

\smallskip 

\section{An ADMM algorithm}\label{sec: main results}

The ADMM approach (e.g., see \cite{boyd2011distributed}) is based on rewriting the optimization in \eqref{eq: log-det0} as 
\begin{equation*}
\begin{array}{ll}
\underset{L,Z\, \in \mathcal{S}_p^{++}}{\text{minimize  }} & \Tr(SL \Theta_x L) - 2\log\det(L) + \lambda \| Z \|_{1,\text{off}}\\
\text {subject to } & L-Z=0. 
\end{array}
\end{equation*}
Define the augmented Lagrangian as 
\begin{align*}
    \mathcal{L}(L,Z,\Lambda)=f(L)+h(Z)+\langle \Lambda, L-Z\rangle +\frac{\rho}{2}\|L-Z\|_F^2, 
\end{align*}
where $\langle A, B\rangle\triangleq \Tr(AB^\transpose)$ and $\Lambda=\Lambda^\top$ is the Lagrange multiplier. The parameter $\rho > 0$ is user-defined and affects only the convergence rate of the ADMM. 
Further, define
\begin{subequations}
    \begin{align}
        f(L) &= \Tr(SL \Theta_x L) - 2\log\det(L) \label{eq: first} \\
        h(Z) &= \lambda \| Z \|_{1,\text{off}}\label{eq: second}. 
    \end{align}
\end{subequations}
We can find the ADMM iterates by sequentially minimizing the augmented Lagrangian with respect to $(L, Z, \Lambda)$. Thus,
\begin{equation}
\begin{aligned}
&\begin{aligned}
& L^{k+1}=\underset{L \in \operatorname{S}_p^{++}}{\arg \min }\, f(L)+\langle \Lambda^k, L-Z^k\rangle +\frac{\rho}{2}\left\|L-Z^k\right\|_F^2 \\
& Z^{k+1}=\underset{Z \in \operatorname{S}_p^{++}}{\arg \min }\,h(Z)+\langle \Lambda^{k}, L^{k+1}\!-\!Z\rangle +\frac{\rho}{2}\left\|L^{k+1}\!-\!Z\right\|_F^2
\end{aligned}\\
&\Lambda^{k+1}=\Lambda^k+\rho\left(L^{k+1}-Z^{k+1}\right). \label{eq: aug lagrangian}
\end{aligned}
\end{equation}

Note that to implement the ADMM iterates above we need to solve only the minimization problems associated with the primal ($L$) and dual variables ($Z$). The solution to the dual minimization is given by the soft-thresholding
\begin{align}
    Z^{k+1}=\frac{1}{\rho}\operatorname{soft}(\rho L^{k+1}+\Lambda^k,\lambda), \label{eq: dual step}
\end{align}
where $\operatorname{soft}(\cdot, \cdot)$ (applied element-wise in \eqref{eq: dual step}) is 
\begin{equation*}
\operatorname{soft}(a, b)= \begin{cases}a-b, & \text { if } a>b \\ 0, & \text { if }|a| \leq b \\ a+b, & \text { if } a<-b.\end{cases} 
\end{equation*}

To solve the primal minimization (associated with $L$) in \eqref{eq: aug lagrangian} set the derivative of the loss to zero to get
\begin{align*}
    \nabla_L (f(L)+\langle \Lambda^k, L-Z^k\rangle +(\rho/2)\left\|L-Z^k\right\|_F^2)&=0\\
    2\Theta_xLS-2L^{-1}+\Lambda^k+\rho(L-Z^k)&=0\\
    2LSL-2\Theta_x^{-1}+\Theta_x^{-1}\Lambda^kL +\rho \Theta_x^{-1}(L-Z^k)L&=0.
\end{align*}
We obtain the last equality by pre-multiplying both sides of the second equation by $\Theta_x^{-1}$ and then post-multiplying by $L$. Collecting the quadratic and linear terms of $L$ finally gives us the quadratic matrix equation
\begin{align}\label{eq: MQE}
        (2LSL+\rho \Theta_x^{-1}L^2)+\Theta^{-1}_x(\Lambda^k-Z^k)L-2\Theta^{-1}_x&=0.
\end{align}
Without loss of generality, we may assume \( \Theta_x = I \), since the data can always be whitened using the transformation \( \Theta_x^{-1/2} \). As a result, the equality in \eqref{eq: MQE} simplifies to the non-symmetric algebraic Riccati equation (NARE): 
\begin{align}\label{eq: NARE}
    L(2S+\rho I)L+(\Lambda^k-\rho Z^k)L-2I=0. 
\end{align}
We discuss an interative method to solve the NARE in \eqref{eq: NARE} using Newton's method in Section \ref{sec: newton}. For the moment, define the functional routines $\texttt{NARE}()$ that solves for $L$ in \eqref{eq: NARE} and $\texttt{soft\_thresholding}()$ that implements \eqref{eq: dual step}. Our ADMM algorithm is
summarized in Alg.~\ref{alg: ADMM recipe}. To ensure that \( L^{k+1}\) in Alg.~\ref{alg: ADMM recipe} remains positive definite, we symmetrize the result and apply a small positive diagonal perturbation.

\begin{algorithm}
\SetAlgoLined
\KwIn{Sample covariance matrix $S$ and $\lambda, \rho > 0$}
\KwOut{Converged values of $L$, $Z$, and $\Lambda$}
Set $k = 0$ and initialize $Z^0, \Lambda^0, B^0 = I$\;
\Repeat{until convergence}{
    $L^{k+1} \leftarrow \texttt{NARE}(S, \Lambda^k, Z^k, \rho)$\;
    $Z^{k+1} \leftarrow \texttt{soft\_thresholding}(L^{k+1}, \Lambda^k, \rho)$\;
    $\Lambda^{k+1} \leftarrow \Lambda^k + \rho (L^{k+1} - Z^{k+1})$\;
}
\caption{ADMM for solving MLE in \eqref{eq: log-det0}}\label{alg: ADMM recipe}
\end{algorithm}

\vspace{-2.0mm}
\subsection{Newton's method for implementing $\texttt{NARE}(\cdot)$}\label{sec: newton}
NARE of the form in~\eqref{eq: NARE} can be solved iteratively using Newton’s method. We briefly outline the approach, following the exposition in~\cite{bini2011numerical}. Consider the matrix-valued nonlinear equation \( F(X) = 0 \), where \( F: \mathcal{V} \to \mathcal{V} \) and \( \mathcal{V} = \mathbb{R}^{m \times n} \). The recurrence relation forms the basis for the newton method to solve for the unknown matrix $X \in \mathbb{R}^{m \times n}$ 
\[
X_{t+1} = X_t - [F'(X_t)]^{-1} F(X_t),
\]
where \( F'(X_t) \) is the Fréchet derivative. Rather than inverting \( F'(X_t) \), one
solves $F'(X_t) H_t = -F(X_t)$ for the increment \( H_t \), updating the solution as \( X_{t+1} = X_t + H_t \).

Consider the function \( F(L) = L\tilde{S}L + (\Lambda^k - \rho Z^k)L - 2I \) for the NARE defined in \eqref{eq: NARE}, where $\tilde{S}=(2S+\rho I)$. Then, the Newton's increment \( H_t \) needed to update \( L_{t+1} = L_t + H_t \) is obtained by solving the Sylvester equation
\begin{align}\label{eq: Sylvester}
\begin{split}
    ((\Lambda^k-\rho Z^k)+L_t\tilde{S})H_t+H_t(L_t\tilde{S})=F(L_t). 
\end{split}
\end{align}
The Sylvester equations are linear in \( H_t \), and standard methods exist for solving them efficiently. Putting the components together, the \texttt{NARE}$(S, \Lambda^k, Z^k, \rho)$ routine implements Newton's method, with each increment computed by solving~\eqref{eq: Sylvester}.

 \begin{remark}\textit{\textbf{(ADMM for ICE and MLE)}}\label{rem: ADMM comparision} The sparse ICE produces a sparse estimate of \( \Theta_y \) by applying the $\ell_1$-norm to $\Theta_y$ \cite{friedman2008sparse}. Instead the MLE in \eqref{eq: log-det0} estimates \( L \) in \( \Theta_y = L \Theta_x L \) for a known $\Theta_x$ by applying the $\ell_1$-norm to  $L$. Consequently, the ADMM for ICE and our MLE appear methodologically similar, except for a crucial difference in computing primal update (line 3) of Algorithm \ref{alg: ADMM recipe}. For ICE, we need to solve the matrix quadratic equation:
\begin{equation}\label{eq: naive MQE}
    \rho\Theta_y^2+\Theta_yM=2I, 
\end{equation}
where \( M=\rho(\Lambda^k-Z^k)-S \). Fortunately, \eqref{eq: naive MQE} is a rare case in which one can solve for \( \Theta_y \succ 0\) in a closed form (see \cite[Section 6.1]{boyd2011distributed} and \cite{higham2000numerical} for details). 

On the other hand, constructing closed-form solutions for the NAREs in \eqref{eq: NARE} and \eqref{eq: MQE} is far from trivial, as the coefficient matrices must satisfy stringent conditions. Ref.~\cite{bini2011numerical} provides an up-to-date survey of such conditions for NAREs arising in transportation problems, under which one might construct closed-form solutions. Unfortunately, these conditions do not hold in our setting, necessitating new theoretical developments for our NARE, which we leave for future work. \QEDW 

\end{remark}

\section{Simulations}\label{sec: simulations}

We analyze the performance of the ADMM algorithm in Alg.~\ref{alg: ADMM recipe} on synthetic networks and two real-world benchmarks: the IEEE 33-bus power system and Bellingham water distribution system. Additionally, we compare its performance with that of CVX \cite{diamond2016cvxpy}, as suggested by \cite{rayas2022learning}.

We consider two metrics, averaged over {\color{black}{40}} instances.  The first metric is the F-score  
\[
\text{F-score} = \frac{2\text{TP}}{2\text{TP} + \text{FP} + \text{FN}} \in [0,1],
\]
where $\text{TP}$ (true positives) is the total number of correctly recovered edges; $\text{FP}$ (false positives) is the total number of incorrectly recovered edges; and finally, $\text{FN}$ (false negatives) is the number of actual edges that were not recovered. A higher F-score indicates better performance in recovering the true structure, with $\text{F-score} = 1$ signifying perfect structure recovery. The second metric is the worst-case error evaluated using $\|\hat{L}-L^*\|_{\infty}$, where $\|A\|_\infty=\sum_{j=1}^n\max_{i,j}|a_{i,j}|$, and $\hat{L}$ is the estimate returned by the ADMM or CVX. 

We plot the above error metrics versus the rescaled sample size $\tau=N/(d^2\log(p))$, where $d = \max_{1\leq i \leq p} \sum_j e_{ij}$ is the maximum degree of $L^{\ast}$ with $e_{ij} = 1$ if $L^*_{ij} \neq 0$, and $e_{ij} = 0$ else. This scaling is justified in \cite{rayas2022learning} which shows that the MLE estimator in \eqref{eq: log-det0} requires $N\approx O(d^2\log(p))$ samples\footnote{We suppress constants in our big-O notation.} to accurately recover the sparsity of $L^*$ with high probability.  
Intuitively, this result means (for a fixed size $p$), the sample complexity is determined by the degree $d$. Therefore, plotting error metrics versus the rescaled sample teases out the effect of $d$ while ignoring the absolute number of samples $N$.  

In light of the foregoing discussion on the sample complexity, we expect that the performance of the MLE \eqref{eq: log-det0}, implemented using Alg.~\ref{alg: ADMM recipe} or the CVX-based approach to exhibit a sharp transition around \( \tau = 1 \). In other words, we anticipate better structure recovery performance for \( \tau \geq 1 \) compared to \( \tau < 1 \). This sharp transition behavior serves as evidence of the (qualitative) accuracy of the algorithms.

\textit{\textbf{Synthetic networks}}: Our first set of simulations evaluates the performance of Alg.~\ref{alg: ADMM recipe} on synthetic networks with a fixed size of $p = 30$.  These networks include Erdős-Rényi, small-world (Watts-Strogatz), and scale-free (Barabási-Albert) networks, with maximum degrees $d=\{4,3,9\}$, respectively. Additionally, we consider a synthetic grid graph ($d=4$) by connecting each node to its fourth-nearest neighbor. 

\textit{\textbf{Benchmark networks}}: Each network has an associated ground truth matrix given by $L^{\ast} = A + \epsilon I_{p}$, where $\epsilon > 0$ and $A$ is the adjacency matrix that defines the edge structure of the networks below. Thus, the constructed Laplacian $L^*$ (i) captures the edge structure and (ii) is positive definite.  

 \noindent \textit{1) Power network:} We consider the IEEE 33-bus power distribution network, whose topology data is publicly available\footnote{\url{https://www.mathworks.com/matlabcentral/fileexchange/73127-ieee-33-bus-system}}. Using this dataset, we construct an adjacency matrix $A$. The corresponding network consists of 33 buses and 32 branches (edges) with a maximum degree of $d = 3$.  
\smallskip  

\noindent \textit{2) Water distribution network:} We examine the Bellingham water distribution network whose raw data files are publicly accessible\footnote{\url{https://www.uky.edu/WDST/index.html}}. The ground truth adjacency matrix $A$, containing 121 nodes and 162 edges with a maximum degree of $d = 6$, is generated by loading the raw data files into the WNTR simulator\footnote{\url{https://github.com/USEPA/WNTR}}, see \cite{seccamonte2023bilevel} for details. 

{\color{black}\textit{\textbf{Algorithmic parameters}}: 
The ADMM algorithm in Alg.\ref{alg: ADMM recipe} terminates when the primal and dual residuals fall below their respective tolerances. The primal residual is $r^k = \|L^k - Z^k\|_F$ and the dual residual is $s^k = \|\rho (Z^k - Z^{k-1})\|_F$. Convergence is declared if $r^k < \varepsilon_{\text{pri}}^k$ and  $s^k < \varepsilon_{\text{dual}}^k$, 
where $\varepsilon_{\text{pri}}^k$ and $\varepsilon_{\text{dual}}^k$ are set according to the guidelines in \cite{boyd2011distributed}. See Section~\ref{sec: CVX solver} for parameters of the CVX-based method.  

The regularization parameter \(\lambda\) was selected via a thorough grid search over \([0,1]\), optimizing for the highest F-score. The best \(\lambda\) for different networks are in Table~\ref{tab:graph_comparison}. Finally, we zeroed out the entries of \( \hat{L} \) returned by ADMM (same for CVX) if their magnitude was below \( 0.01 \). This hard thresholding is a standard practice for recovering the network.



\subsection{Results and discussion}

In Figs.~\ref{fig: synthetic data scaling_small},~\ref{fig: synthetic data scaling_large}, and~\ref{fig: real data scaling}, for various values of \( \tau \), we compute \( N \) and generate synthetic data \( \{y_t\}_{t=1}^N \) using \( y_t = (L^*)^{-1} x_t \), where \( \{x_t\} \) is a zero mean Gaussian with the inverse covariance matrix $\Theta_x$. The sample covariance matrix \( S \) is then computed as \( S = N^{-1} \sum_t y_t y_t^\top \).

Figs.~\ref{fig: synthetic data scaling_small} and \ref{fig: synthetic data scaling_large} shows the recovery performance of synthetic networks as a function of \( \tau \). The F-score increases with \( \tau \) for all networks. Grid and small-world networks achieve near-perfect F-score (\(\geq 0.99\)) around \( \tau \approx 1.5 \). In contrast, the scale-free network in Fig.~\ref{fig: synthetic data scaling_large} requires slightly more samples for high accuracy due to its high-degree hub nodes ($d=10$), but exceeds 0.95 F-score for \( \tau \approx 2 \). Similar inferences hold for the worst case metric (bottom panels of Figs.~\ref{fig: synthetic data scaling_small} and \ref{fig: synthetic data scaling_large}). Fig.~\ref{fig: real data scaling} highlights the recovery performance of benchmark networks as a function of \( \tau \). The F-score increases with \( \tau \) while the worst-case error decreases with $\tau$.

The qualitative performance of the ADMM and CVX approaches is nearly identical, which is to be expected because both the methods solve a convex optimization problem.  Table~\ref{tab:graph_comparison} presents a comparison of the ADMM and CVX-based approaches. \textbf{\textit{Both achieve similar F-scores, however, ADMM significantly outperforms CVX in iteration count and execution time.}} For e.g., in the water network case---the largest network considered---our ADMM converged in 24 iterations within 3 seconds, whereas the CVX approach required 150 iterations and 85 seconds.



\begin{figure}
\centering
\includegraphics[width=0.82\linewidth]{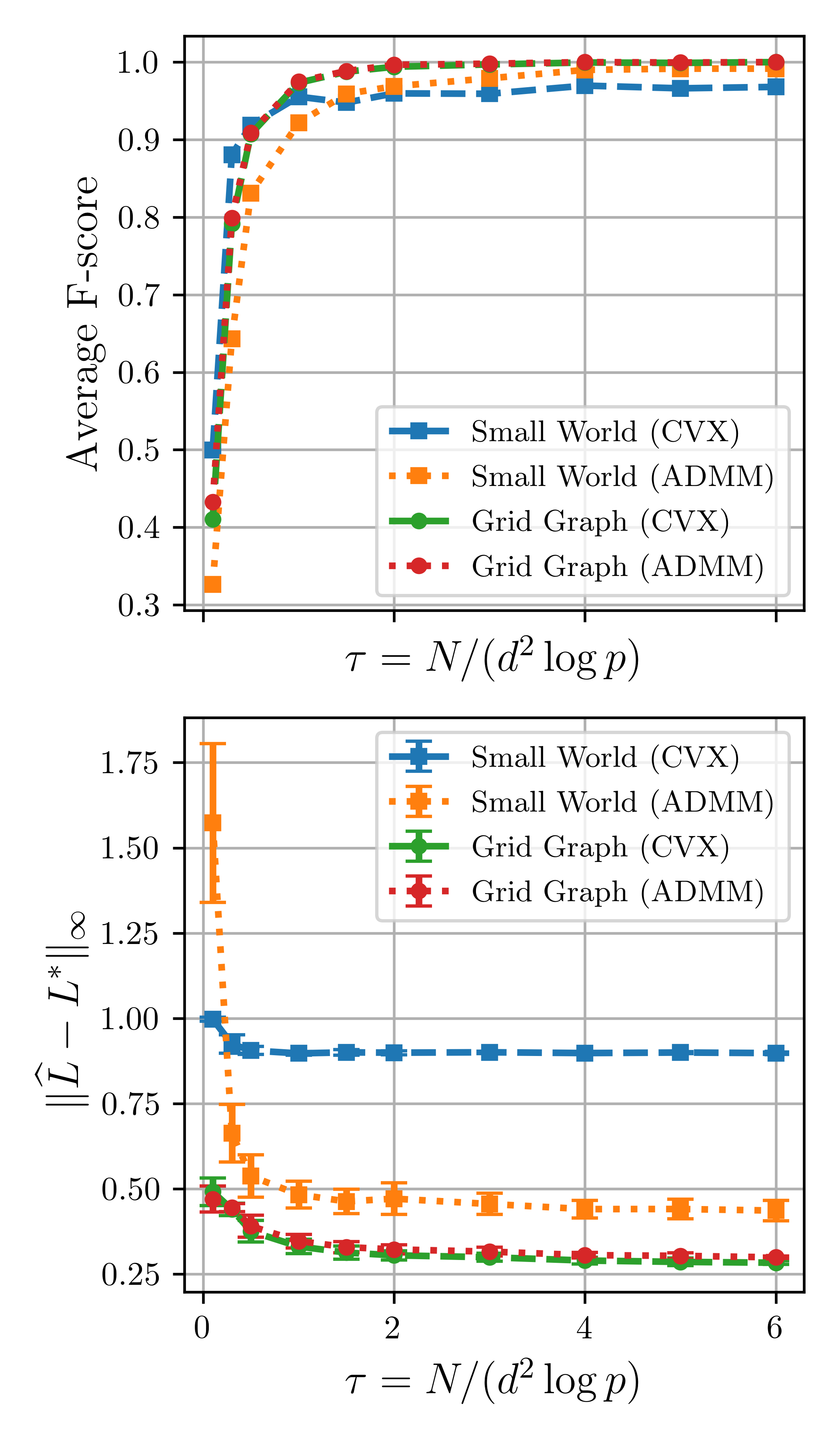}
\caption {\small {\bf Synthetic network recovery (small degree)}:  Top: Average F-score versus rescaled sample size $N / (d^2 \log p)$. Bottom: Worst-case $\ell_\infty$ error across off-diagonal entries. Both the ADMM- and CVX-based methods achieve comparable performance.}\label{fig: synthetic data scaling_small}
\end{figure}

\begin{figure}
\centering
\includegraphics[width=0.82\linewidth]{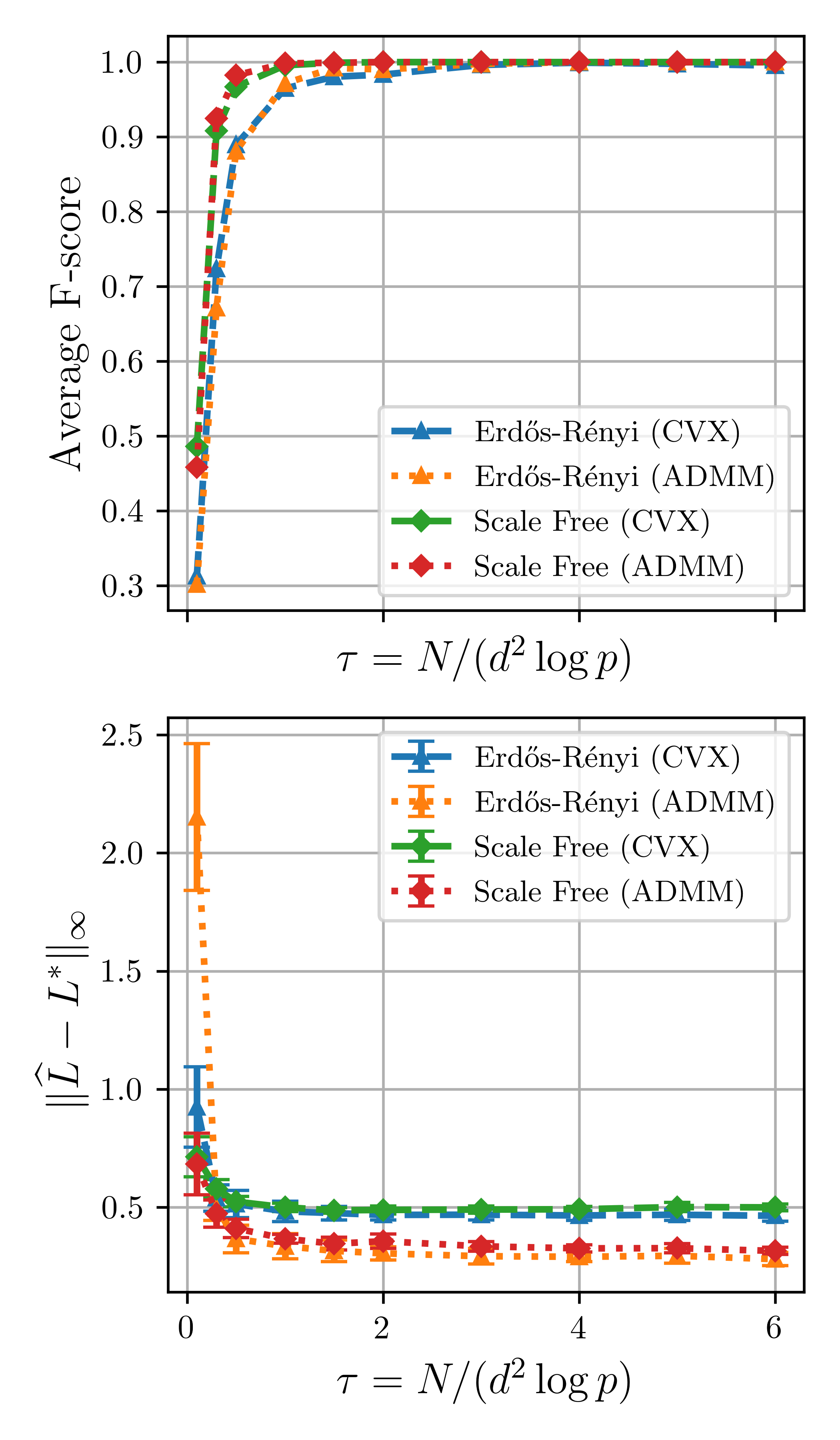}
\caption {\small {\bf Synthetic network recovery (large degree)}:  Top: Average F-score versus rescaled sample size $N / (d^2 \log p)$. Bottom: Worst-case $\ell_\infty$ error across off-diagonal entries. Both the ADMM- and CVX-based methods achieve comparable performance.
}\label{fig: synthetic data scaling_large}
\end{figure}

\begin{figure}
\centering
\includegraphics[width=0.82\linewidth]{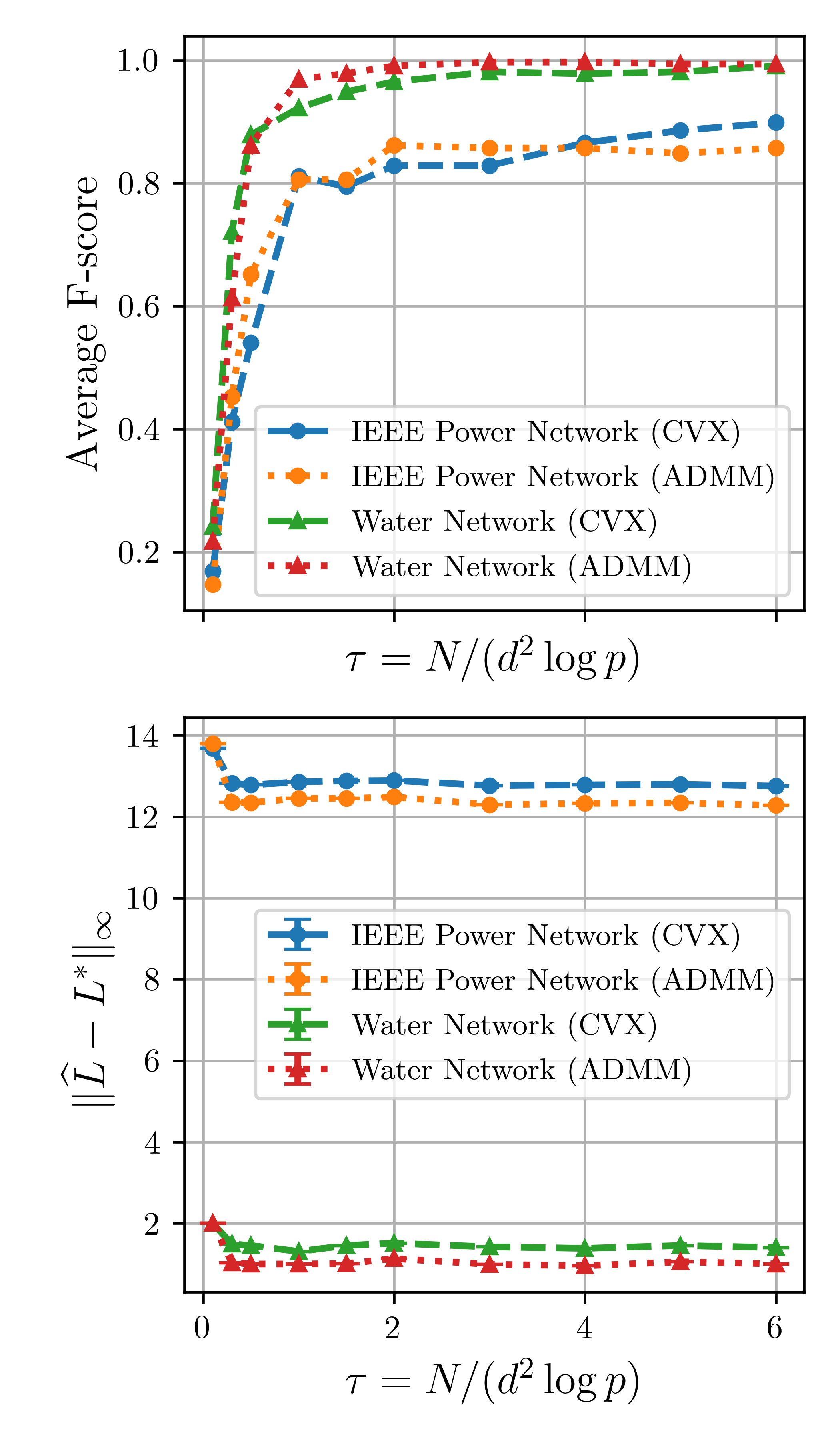}
\caption {\small {\bf Benchmark networks recovery}:  Top: Average F-score versus rescaled sample size $N / (d^2 \log p)$. Bottom: Worst-case $\ell_\infty$ error across off-diagonal entries. Both the ADMM- and CVX-based algorithms achieve excellent F-scores for the water network, but perform less well for the IEEE power network.
}\label{fig: real data scaling}
\end{figure}


\begin{table*}[h]
    \centering
    \renewcommand{\arraystretch}{1.5} 
    \setlength{\tabcolsep}{9pt} 
        \caption{\text{ADMM vs CVX approach}}
    \begin{tabular}{|l|c|c|c|c|c|c|}
        \hline
        \textbf{Network} & \makecell{\textbf{Best} $\lambda$ \\ (ADMM / CVX)} & \makecell{\textbf{Best} F-score \\ (ADMM / CVX)} & \makecell{\textbf{Samples} {$N$}\\ $(\tau=4)$} & \makecell{\textbf{Iterations} \\ (ADMM / CVX)} & \makecell{\textbf{Time (sec)} \\ (ADMM / CVX)} \\
        \hline
        Small world ($p=30$,  $d=4$)  & 0.333 / 0.889 & 0.951 /\textbf{1.000} & $54 $ & $\mathbf{83}$ / $325$ &  $\mathbf{0.20}$ / $1.16$ \\
        \hline
        Erdős-Rényi ($p=30$,  $d=5$) & 0.222 / 0.220 & \textbf{0.984} / 0.969 & $85$ & $\mathbf{75}$ / $450$ & $\mathbf{0.15}$ / $1.55$ \\
        \hline
        Scale free ($p=30$  $d=10$)  & 0.111 / 0.111 & \textbf{1.000} / 0.983 & $340$ & $\mathbf{56}$ / $525$ &  $\mathbf{0.25}$ / $1.84$\\
        \hline
        Grid  ($p=30, d=4$)  & 1.000 / 0.444     & 0.950 / 0.950  & $54$ & $\mathbf{265}$ / $275$ &  $\mathbf{0.77}$ / $0.94$\\
        \hline
        IEEE power ($p=33$,  $d=3$)  & 0.333 / 0.333 & \textbf{0.831} / 0.821 & $31$ & $\mathbf{36}$ / $200$ &  $\mathbf{0.323}$ / $1.10$\\
        \hline
        Water ($p=120$, $d=5$) & 0.222 / 0.222 & \textbf{0.954} / 0.949 & $119$ & $\mathbf{24}$ / $150$ &  $\mathbf{3.01}$ / $85.83$\\
        \hline
    \end{tabular}
    \label{tab:graph_comparison}
\end{table*}

\subsection{Reproducible results and other limitations} \label{sec: CVX solver}
All experiments were implemented in Python 3.11.5 on a Windows 11 machine with an Intel Core i7-8665U CPU @ 1.90GHz (4 cores and 8 threads). The CVX-based method was implemented following the repository\footnote{\url{https://github.com/AnirudhRayas/SLNSCL}} provided in \cite{rayas2022learning}, using CVXPY 1.6.4 with the SCS solver v3.2.7.post2 and SciPy 1.15.2. For SCS solver we set the convergence tolerance as $1e-4$. The results presented in this paper are fully reproducible using our code available in the repository\footnote{\url{https://github.com/rohithmada00/NARE_graphical_IID}}.  


{\color{black} A potential criticism of our ADMM approach is the lack of multiple comparative methods. Since the MLE \eqref{eq: log-det0} is new, we are unaware of existing methods for comparison except for the CVX-based approach in \cite{rayas2022learning}. While we could have included the Graphical LASSO-based two-hop method, we consider this unnecessary, as \cite{rayas2022learning} demonstrated that the MLE in \eqref{eq: log-det0} outperforms the two-hop method. Our contribution lies in providing an efficient algorithm for solving \eqref{eq: log-det0}. However, motivated by recent advances in ICE problem \cite{eftekhari2024algorithm}, we plan to develop a suite of efficient algorithms for \eqref{eq: log-det0}.}

\section{Conclusions}
This paper addressed the structure learning problem in equilibrium networks, formulated as an $\ell_1$-regularized maximum likelihood estimation problem. As a modest theoretical contribution, we demonstrated that, in the unregularized case, the estimator admitted a pleasing analytical expression.

To solve the learning problem, the paper introduced an ADMM algorithm. While the dual and Lagrange multiplier updates admit a closed form, the primal step lacked such form and required solving a non-symmetric algebraic Riccati equation (NARE). This was handled iteratively using a Newton’s method. Our results showed that, for nearly all networks, the ADMM and CVX solvers achieve comparable performance; however, ADMM converges in fewer iterations and have lower time complexity.

Several directions remain open for future investigation. A key avenue is to establish theoretical convergence rates for the ADMM iterates. Another important challenge is to derive conditions that ensure the existence of a symmetric positive definite solution to the NARE, and to further explore both Newton-based iterative schemes and invariant subspace methods. Finally, incorporating explicit Laplacian constraints, as  in~\cite{kumar2019structured}, remains an interesting extension.

\bibliographystyle{unsrt}
\bibliography{BIB.bib}

\end{document}